\documentclass[reqno]{amsart} 
\usepackage{amsfonts, amsmath, amsthm, amssymb, latexsym, graphicx}

\def\cal{\mathcal }
\theoremstyle{plain}
\newtheorem{theorem}{Theorem}[section]

\newtheorem{lemma}[theorem]{Lemma}

\theoremstyle{definition}

\newtheorem{remark}[theorem]{Remark}

\numberwithin{equation}{section}

\title[Central limit theorems for multivariate Bessel processes]{Functional central limit theorems
 for multivariate Bessel processes in the freezing regime}

\author{Michael Voit, Jeannette H.C. Woerner} 
\address{Fakult\"at Mathematik, Technische Universit\"at Dortmund,
          Vogelpothsweg 87,
          D-44221 Dortmund, Germany}
\email{michael.voit@math.tu-dortmund.de, jeannette.woerner@math.tu-dortmund.de}


\begin{document}
\date{\today}

\begin{abstract}
Multivariate Bessel processes  $(X_{t,k})_{t\ge0}$ describe interacting particle systems of 
Calogero-Moser-Sutherland type and are related with $\beta$-Hermite and $\beta$-Laguerre ensembles. 
They depend on a root system and a multiplicity $k$.
Recently, several limit theorems were derived for $k\to\infty$ with fixed starting point.
Moreover, the SDEs of $(X_{t,k})_{t\ge0}$ were used to derive strong laws of large numbers for
 $k\to\infty$ with starting points of the form $\sqrt k\cdot x$ with $x$ in the interior of the Weyl chambers.
Here we provide associated almost sure functional central limit theorems which are locally uniform in $t$.
The Gaussian limit processes admit explicit representations in terms of
the solutions of  associated deterministic ODEs.
\end{abstract}

\maketitle

\section{Introduction} 
Integrable interacting particle systems on  $\mathbb R$ of 
Calogero-Moser-Suther\-land type with $N$ particles 
can be described as  multivariate Bessel processes on 
 appropriate closed Weyl chambers in  $\mathbb R^N$. These processes 
 are time-homogeneous diffusions with well-known generators of the transition semigroups and  transition probabilities,
and they are solution of the associated stochastic differential equations (SDEs); see  \cite{CGY,GY,R1,R2, 
RV1,RV2,DV, A}. 
These multivariate Bessel processes $(X_{t,k})_{t\ge0}$ are  described via root systems,  a possibly multidimensional
multiplicity parameter $k$ and by their starting points $X_{0,k}:=x$.
The multiplicies $k$ are coupling constants which describe the strength of interaction
of the particles.
 We restrict our attention  to the root systems  $A_{N-1}$, $B_N$, and $D_N$ on  $\mathbb R^N$ as the
 most relevant cases in view of  interacting particle systems and random matrix theory. 

We briefly recapitulate the most important cases $A_{N-1}$ and $B_N$.
For $A_{N-1}$, we have a multiplicity $k\in]0,\infty[$, the processes live on the closed Weyl chamber
$$C_N^A:=\{x\in \mathbb R^N: \quad x_1\ge x_2\ge\ldots\ge x_N\},$$
and  the generator of the transition semigroup is 
\begin{equation}\label{def-L-A} Lf:= \frac{1}{2} \Delta f +
 k \sum_{i=1}^N\Bigl( \sum_{j\ne i} \frac{1}{x_i-x_j}\Bigr) \frac{\partial}{\partial x_i}f ,
 \end{equation}
where we assume reflecting boundaries, i.e., the domain of $L$ is
$$D(L):=\{f|_{C_N^A}: \>\> f\in C^{(2)}(\mathbb R^N), \>\>\> f\>\>\text{ invariant under all coordinate permutations}\}.$$

For $B_N$, we have  the multiplicity $k=(k_1,k_2)\in]0,\infty[^2$, the processes live on 
$$C_N^B:=\{x\in \mathbb R^N: \quad x_1\ge x_2\ge\ldots\ge x_N\ge0\},$$
and the generator  is 
\begin{equation}\label{def-L-B} Lf:= \frac{1}{2} \Delta f +
 k_2 \sum_{i=1}^N \sum_{j\ne i} \Bigl( \frac{1}{x_i-x_j}+\frac{1}{x_i+x_j}  \Bigr)
 \frac{\partial}{\partial x_i}f 
\quad + k_1\sum_{i=1}^N\frac{1}{x_i}\frac{\partial}{\partial x_i}f, \end{equation}
where we again assume reflecting boundaries.

 By \cite{R1,R2,RV1,RV2}, the transition probabilities of   $(X_{t,k})_{t\ge0}$ have the form
\begin{equation}\label{density-general}
K_t(x,A)=c_k \int_A \frac{1}{t^{\gamma_k+N/2}} e^{-(\|x\|^2+\|y\|^2)/(2t)} J_k(\frac{x}{\sqrt{t}}, \frac{y}{\sqrt{t}}) 
\cdot w_k(y)\> dy
\end{equation}
for $t>0$,  $x\in C_N$, and  $A\subset C_N$ a Borel set.
For the root systems $A_{N-1}$ and $B_N$, we here have the  weight functions $w_k$ of the form
\begin{equation}\label{def-wk}
w_k^A(x):= \prod_{i<j}(x_i-x_j)^{2k}, \quad\quad 
w_k^B(x):= \prod_{i<j}(x_i^2-x_j^2)^{2k_2}\cdot \prod_{i=1}^N x_i^{2k_1},\end{equation}
and $\gamma_k^A(k):=kN(N-1)/2$ and   $\gamma_k^B(k_1,k_2):=k_2N(N-1)+k_1N$
respectively. 
$w_k$ is homogeneous of degree $2\gamma_k$, and
$c_k>0$ is a known normalization, and
$J_k$  a multivariate Bessel function of type $A_{N-1}$ or $B_N$ with multiplicities $k$ or $(k_1,k_2)$ 
which is analytic on $\mathbb C^N \times \mathbb C^N $ with
$ J_k(x,y)>0$ for $x,y\in \mathbb R^N $.
Moreover, $J_k(x,y)=J_k(y,x)$ and $J_k(0,y)=1$
for  $x,y\in \mathbb C^N $. For this and further informations we refer  e.g. to \cite{R1,R2}.
In particular, if $X_{0,k}=0$, then $X_{t,k}$ has the Lebesgue density
\begin{equation}\label{density-A-0}
 \frac{c_k}{t^{\gamma+N/2}} e^{-\|y\|^2/(2t)} \cdot w_k(y)\> dy
\end{equation}
on $C_N$ for $t>0$.
Hence, for the root systems $A_{N-1}$, $B_N$, these  processes $(X_{t,k})_{t\ge0}$ (with proper  parameters)  are related to 
random matrix theory, as the distributions in (\ref{density-A-0})
are those of  the ordered eigenvalues
of  $\beta$-Hermite and $\beta$-Laguerre ensembles; see e.g.~the tridiagonal random matrix models of
 Dumitriu and Edelman \cite{DE1, DE2}.  For other starting points, there exist
further tridiagonal random matrix models for the distributions of 
 $X_{t,k}$; see \cite{AG, HP}.
 
We here study limit theorems for  $(X_{t,k})_{t\ge0}$ when 
 one or several components of $k$ tend to $\infty$ in a coupled way 
(we  briefly write $k\to\infty$ by misuse of notation).
In physics, this means freezing, and 
in  random matrix theory, $k\to\infty$ roughly means $\beta\to\infty$ in \cite{DE2}.
Several  limit  theorems for $k\to\infty$ are known for fixed
 starting points $x\in\mathbb R^N $; see \cite{AKM1, AKM2, AM, V}
where for  $x=0$ the results fit to  \cite{DE2}.
 In the present paper we
  regard the $(X_{t,k})_{t\ge0}$  as solutions of the SDEs
	\begin{equation}\label{SDE-general}
dX_{t,k}= dB_t +  \frac{1}{2} (\nabla(\ln w_k))(X_{t,k}) \> dt
\end{equation}
with starting points $X_{0,k}=\sqrt k\cdot x$ with $x$ in the interior of the  Weyl chambers
 and an $N$-dimensional Brownian motion  $(B_t)_{t\ge0}$. Note that by \cite{CGY}, see also \cite{Sch} and \cite{GM}, under the condition $E\Bigl(\int_0^t \nabla(\ln w_k)(X_{s,k}) \> ds\Bigr)<\infty$, (\ref{SDE-general}) has a unique (strong) solution $(X_{t,k})_{t\ge0}$, the corresponding Bessel process.
Moreover, if all components of $k$ are at least $ 1/2$, and 
 $x$ is in the interior of $C_N$, then  $(X_{t,k})_{t\ge0}$  does not hit the boundary a.s..

The associated renormalized processes $(X_{t,k}/\sqrt k)_{t\ge0}$ then start in  $x$ and satisfy SDEs with drifts independent of $k$ and vanishing
  diffusion part as $k\to\infty$. In \cite{AV}, several strong, locally uniform limit theorems, of law of large number type, were derived.
Moreover, only in a simple special case, a corresponding central limit theorem was derived in \cite{AV}.
In the present paper we derive  functional central limit theorems in general as follows:
The renormalized processes $(X_{t,k}/\sqrt k)_{t\ge0}$ satisfy SDEs
 with vanishing Brownian parts for $k\to\infty$. Let  $\phi:=(\phi(t,x))_{t\ge0}$ be 
 the solution of the associated deterministic 
limit differential equation  with start in $x$.
We shall present  Gaussian diffusions $W:=(W_t)_{t\ge0}$ such that
 \begin{equation}\label{clt-introd}
\sqrt{k} (\frac{X_{t,k}}{\sqrt k} -\phi(t,x)) \longrightarrow W_t \quad\quad\text{for}\quad k\to\infty
\end{equation}
 locally uniformly in $t$ with  rate $O(1/\sqrt k)$ a.s..
This  $\phi$  plays an essential role in these limit theorems, especially in the covariance matrix of $W$.
Unfortunately, up to  particular examples,  $\phi$ cannot be written down explicitly. However,
  $\phi$ can be described explicitly in some new coordinates using elementary symmetric polynomials;
we shall  present these  results on $\phi$ in \cite{VW}.
We  mention that for particular starting points $x\in C_N$, the functions $\phi$ are provided explicitly
via the  zeros of the Hermite  or  Laguerre polynomial. This connection with  zeros of classical
orthogonal polynomials 
 already appeared in  \cite{AKM1, AKM2, AM, AV}.

This paper is organized as follows. In the next two sections we study Bessel processes of type $A_{N-1}$ 
for $k\to\infty$ and, in the  $B_N$-case, for
 multiplicities  $k=(k_1,k_2)=(\nu\cdot\beta, \beta)$ with $\nu>0$  and $\beta\to\infty$.
Depending on $\nu$, here the zeros of the Laguerre polynomial $L_N^{\nu-1}$ play a prominent role.
 Section 4 is devoted to the root systems $D_N$.
These results are then used in Section 5 to settle also the limits  $k=(k_1,k_2)$
for $k_1\ge0$ fixed and  $k_2\to\infty$ in the $B_N$-case. As in the $B_N$-case
  the case  $k=(k_1,k_2)$
for $k_2\ge0$ fixed and  $k_1\to\infty$ was already treated in \cite{AV},
 we skip this case here.
In Section 6 we consider an extension of our results by adding an additional drift 
of the form $-\lambda X_{t,k}$, $\lambda \in\mathbb R$
to our SDEs.
For $\lambda>0$ the resulting process is ergodic and mean reverting. For $N=1$ with $\lambda>0$,
 it is related to the Cox-Ingersoll-Ross process in  finance.

\section{The root system $A_{N-1}$}

The SDE (\ref{SDE-general}) for  Bessel processes $( X_{t,k})_{t\ge0}$
of type $A_{N-1}$  reads as
\begin{equation}\label{SDE-A}
 dX_{t,k}^i = dB_t^i+ k\sum_{j\ne i} \frac{1}{X_{t,k}^i-X_{t,k}^j}dt \quad\quad(i=1,\ldots,N).
\end{equation}
with an $N$-dimensional Brownian motion $(B_t^1,\ldots,B_t^N)_{t\ge0}$.
Hence, the  renormalized processes $(\tilde X_{t,k}:=X_{t,k}/\sqrt k)_{t\ge0}$ satisfy
\begin{equation}\label{SDE-A-normalized}
d\tilde X_{t,k}^i =\frac{1}{\sqrt k}dB_t^i + \sum_{j\ne i} 
 \frac{1}{\tilde X_{t,k}^i-\tilde X_{t,k}^j}dt\quad\quad(i=1,\ldots,N). 
\end{equation}
The solutions of (\ref{SDE-A-normalized})  are closely related to limit  $k=\infty$ in Lemma 2.1 of \cite{AV}. By \cite{AV} we know that for each starting point $x_0$ in
 the interior of the Weyl chamber $C_N^A$,  the dynamical system 
\begin{equation}\label{deterministic-boundary-A}
\frac{dx}{dt}(t) =H(x(t))
\end{equation}
with
$$H(x):=
\Bigl( \sum_{j\ne1} \frac{1}{x_1-x_j},\ldots,\sum_{j\ne N} \frac{1}{x_N-x_j} \Bigr)$$
has a unique solution  $\phi(t,x_0)$ for $t\ge0$.
It admits an explicit solution
\begin{equation}\label{special-solution}
\phi(t,c\cdot {z})= \sqrt{2t+c^2}\cdot {z}
\end{equation}
for special starting points of the form $cz$, with $c>0$ and
\begin{equation}\label{char-zero-A}
z:=(z_1,\ldots,z_N)\in C_N^A
\end{equation}
 consisting of
 the ordered zeroes $z_i$  of Hermite polynomial $H_N$. We assume  that   the $(H_N)_{N\ge 0}$ are orthogonal w.r.t.
 the density  $e^{-x^2}$ as in \cite{S}. 

It can be shown that the solutions $\phi$ in (\ref{special-solution}) are attracting in some way, and
all solutions of the ODEs (\ref{deterministic-boundary-A}) can be determined explicitly after
 some transformation of coordinates; we shall discuss this in  \cite{VW}.

We here only discuss the growth behavior of  $\phi$. We observe that
\begin{align}\label{derivative-2-a}
\frac{d}{dt}\|\phi(t,x)\|_2^2 &= 2\sum_{i=1}^N \phi_i(t,x)\cdot \dot \phi_i(t,x) \notag\\
 &= 2\sum_{i,j=1,\ldots,N, i\ne j} \frac{\phi_i(t,x)}{\phi_i(t,x)-\phi_j(t,x)} \> =\> N(N-1). 
\end{align}
As $\|\phi(0,x)\|_2^2 =\|x\|^2$, we see that for all $t\ge0$ and all $x$,
\begin{equation}\label{ODE-norm}
\|\phi(t,x)\|_2^2 =N(N-1)t+\|x\|^2.
\end{equation}

We now use  $\phi$  to derive limit theorems for Bessel processes of type 
$A_{N-1}$ for $k\to\infty$. We have the 
  following strong limit law from Theorem 2.4 of  \cite{AV}.

\begin{theorem}\label{SLLN-A} Let $x$ be a point in the interior of $C_N^A$ and $y\in \mathbb R^N$. 
Let $k_0\ge 1/2$ such that $\sqrt k \cdot x+y$ is in the interior of $C_N^A$ for $k\ge k_0$.

For $k\ge k_0$ consider the Bessel processes $(X_{t,k})_{t\ge0}$ of type $A_{N-1}$ starting  at
$\sqrt k\cdot x+y$.
Then, for all $t>0$,
\begin{equation}\label{diff-process-A}
  \sup_{0\le s\le t, k\geq k_0}\|X_{s,k}-\sqrt k \phi(s,x) \|<\infty \quad\quad a.s..
\end{equation}
In particular, locally uniformly in $t$ a.s.,
$X_{t,k}/\sqrt k\to \phi(t,x)$ for $k\to\infty.$
\end{theorem}

We now turn to an associated functional central limit theorem. For this we again fix  $x$
 in the interior of $C_N^A$  and consider  $t\mapsto \phi(t,x)$ ($t\ge0$) as solution of (\ref{deterministic-boundary-A}). We also introduce the $N$-dimensional process 
$(W_t)_{t\ge0}$ as unique solution of the inhomogeneous linear SDE
 \begin{equation}\label{SDE-W1-A}
dW_t^i =dB_t^i+ \sum_{j\ne i} \frac{W_t^j-W_t^i}{(\phi_i(t,x)-\phi_j(t,x))^2}dt \quad\quad(i=1,\ldots,N).
\end{equation}
with initial condition $W_0=0$; notice that  the denominator is $\ne0$ for $t>0$. In matrix notation,
(\ref{SDE-W1-A}) means that
\begin{equation}\label{SDE-W2-A}
dW_t=dB_t+A(t,x)W_tdt 
\end{equation}
with the matrices $A(t,x)\in \mathbb R^{N\times N}$ with
$$A(t,x)_{i,j}:= \frac{1}{(\phi_i(t,x)-\phi_j(t,x))^2}, \quad
 A(t,x)_{i,i}:=-\sum_{j\ne i} \frac{1}{(\phi_i(t,x)-\phi_j(t,x))^2}$$
for $i,j=1,\ldots, N$, $i\ne j$. The process $(W_t)_{t\ge0}$ is obviously Gaussian and admits
the explicit representation in terms of matrix-valued exponentials
\begin{equation}\label{SDE-W3-A} 
W_t=e^{\int_0^t A(s,x)ds}\int_0^te^{-\int_0^s A(u,x)du}d B_s \quad(t\ge0).
\end{equation}
  $(W_t)_{t\ge0}$ is related to the Bessel processes $( X_{t,k})_{t\ge0}$ by the following functional CLT:

\begin{theorem}\label{FCLT-A} Let $x$ be a point in the interior of $C_N^A$ and  $y\in \mathbb R^N$.  
Let $k_0\ge 1/2$ such that $\sqrt k \cdot x+y$ is in the interior of $C_N^A$ for $k\ge k_0$.
For $k\ge k_0$ consider the Bessel processes $(X_{t,k})_{t\ge0}$ of type $A_{N-1}$ starting  at
$\sqrt k\cdot x+y$.
Then, for  $t>0$,
\begin{equation}\label{CLT-diff-A}
\sup_{0\le s\le t, k\geq k_0}\sqrt k \cdot \|X_{s,k}-\sqrt k \phi(s,x) -W_s\|<\infty \quad\quad a.s.,
\end{equation}
 i.e., $\lim_{k\to\infty}\sqrt{k} (\frac{X_{t,k}}{\sqrt k} -\phi(t,x)) =W_t$
 locally uniformly in $t$ a.s.~with  rate $O(1/\sqrt k)$.
\end{theorem}

\begin{proof}
For $k\ge k_0$ consider the processes 
$(R_{t,k}:=X_{t,k}- \sqrt k \phi(t,x) -W_t)_{t\ge0}$
 on $\mathbb R^N$. Then
$R_{0,k}=0$, and, by the SDEs (\ref{SDE-W1-A}) and  (\ref{SDE-A}) and the ODE for $\phi$ in (\ref{deterministic-boundary-A}),
\begin{align}
R_{t,k}^i = k \int_0^t \sum_{j\ne i}\Biggl(& \frac{1}{X_{s,k}^i-X_{s,k}^j}-\frac{1}{\sqrt k (\phi_i(s,x)-\phi_j(s,x))}
\notag\\
&-  \frac{W_s^j-W_s^i}{\bigl(\sqrt k (\phi_i(s,x)-\phi_j(s,x))\bigr)^2}\Biggr)ds
\notag
\end{align}
for $i=1,\ldots,N$. 
We now use Taylor expansion for the function $1/x$ with Lagrange remainder around some point $x_0\ne 0$, i.e.,
$$\frac{1}{x} = \frac{1}{x_0}- \frac{x-x_0}{x_0^2} + \frac{(x-x_0)^2}{\tilde x^3}$$
with some $\tilde x$ between $x\ne0$ and $x_0\ne0$ where $x,x_0$ have the same sign. Taking 
$$x=X_{s,k}^i-X_{s,k}^j \quad\text{and}\quad x_0=\sqrt k (\phi_i(s,x)-\phi_j(s,x)),$$
we arrive at
\begin{align}
R_{t,k}^i &= -\int_0^t\Biggl(\sum_{j\ne i}
\frac{(X_{s,k}^i-\sqrt k \phi_i(s,x)-W_s^i)-(X_{s,k}^j-\sqrt k \phi_j(s,x)-W_s^j)}{ (\phi_i(s,x)-\phi_j(s,x))^2}
+H_{s,k}^i\Biggr)ds
\notag\\
&= -\int_0^t\Biggl(\sum_{j\ne i}\frac{R_{s,k}^i-R_{s,k}^j}{ (\phi_i(s,x)-\phi_j(s,x))^2}+H_{s,k}^i\Biggr)ds
\notag
\end{align}
with the error terms
$$H_{s,k}^i=k\sum_{j\ne i}
\frac{\Bigl((X_{s,k}^i-\sqrt k \phi_i(s,x))-(X_{s,k}^j-\sqrt k \phi_j(s,x))\Bigr)^2}{
\Bigl(\sqrt k (\phi_i(s,x)-\phi_j(s,x))+D_{i,j}(s)\Bigr)^3}$$
where, by the Lagrange remainder,
$$|D_{i,j}(s)|\le |(X_{s,k}^j-\sqrt k \phi_j(s,x))-(X_{s,k}^j-\sqrt k \phi_j(s,x))|.$$
By Theorem \ref{SLLN-A}, this can be bounded by some a.s.~finite random variable $D$ independent
 of $i,j$, $s\in[0,t]$, and $k\ge k_0$ where $D$ depends on $x,y,t$. Therefore, 
$$|H_{s,k}^i|\le H/ \sqrt k\quad\quad\text{for}\quad k\ge k_0, \> s\in[0,t], \> i=1,\ldots,N $$
 with some  a.s.~finite random variable $H$. In summary,
$$R_{t,k}=-\int_0^t( A(s,x)R_{s,k} + H_{s,k})\> ds, \quad R_{0,k}=0$$
and thus, for suitable norms and all $u\in[0,t]$, 
$$\|R_{u,k}\|\le A\int_0^u\|R_{s,k}\|\> ds + \frac{t\cdot \|H\|}{\sqrt k}$$
with $A:= \sup_{s\in [0,t]} \|A(s,x)\|<\infty$. Hence, by the classical lemma of Gronwall,
$$\|R_{u,k}\|\le\frac{t\|H\|}{\sqrt k}e^{tA}$$
for all  $u\in[0,t]$. This yields the claim.
\end{proof}

\begin{remark}\label{remark-verhalten-a}
The processes $(X_{t,k})_{t\ge0}$ of type A  admit some algebraic properties
 which are related with corresponding  properties of $\phi$ and $A(t,x)$:
\begin{enumerate}
\item[\rm{(1)}] $(X_{t,k})_{t\ge0}$ has the same scaling as Brownian motions, i.e., for $r>0$, 
 the process $(\frac{1}{r} X_{r^2t, k})_{t\ge0}$ is also a  Bessel process of type A with the same $k$.
  The corresponding relations for $\phi$ and  $A$ are
$$\phi(r^2t, rx)=r\cdot \phi(t,x), \quad A(r^2t, rx)=\frac{1}{r^2}A(t,x)\quad \text{for}\quad
r>0, \> t\ge0.$$
Moreover, for $(W_t)_{t\ge0}$ from (\ref{SDE-W3-A}), 
$ (\frac{1}{r} W_{r^2t})_{t\ge0}$ is also a process of this type where $x$ is replaced by $rx$ in
 Eqs.~(\ref{SDE-W1-A})--(\ref{SDE-W3-A}). 
\item[\rm{(2)}]  By (\ref{SDE-A}), the center of gravity
$$\Bigl(\overline{X_{t,k}} :=\frac{1}{N}(X_{t,k}^1+\ldots +X_{t,k}^N) =\frac{1}{N}(B_{t}^1+\ldots +B_{t}^N) \Bigr)_{t\ge0}$$
is a Brownian motion up to scaling. For $\phi$ and  $(W_t)_{t\ge0}$ this means that
$$\sum_{i=1}^N \phi_i(t,x)=x_1+\ldots+ x_N, \quad W_t^1+\ldots + W_{t}^N=B_{t}^1+\ldots +B_{t}^N$$
for $t\ge0$. This yields that the sums over
 all rows and columns of $A(t,x)$ are equal to 0, i.e., $A(t,x)$ is  singular.
\item[\rm{(3)}] Let $(\widehat{X_{t,k}})_{t\ge0}$ be the orthogonal projection of the Bessel process
 $(X_{t,k})_{t\ge0}$  to the orthogonal complement $(1,\ldots,1)^\bot\subset \mathbb R^N$
of  $(1,\ldots,1)$. Then
$(\widehat{X_{t,k}})_{t\ge0}$ is again a diffusion  living on this $(N-1)$-dimensional subspace 
which is stochastically independent of the center-of-gravity-process  $(\overline{X_{t,k}})_{t\ge0}$
 on $\mathbb R\cdot(1,\ldots,1)$. 
On the level of $\phi$ and $A$, we have the relations
$$\phi(t, x+r\cdot(1,\ldots,1))=\phi(t,x)+r\cdot(1,\ldots,1)$$ 
and $A(t, x+r\cdot(1,\ldots,1))=A(t, x)$ for $r\in\mathbb R$. 
%\item[\rm{(4)}]  The Bessel processes $(X_{t,k})_{t\ge0}$ with start in $0$ 
%have the same time inversion property like Brownian motions;
% see Section 4 of \cite{CGY} for details. We do not know whether this leads to further relations for $\phi$ and $A$,
%and whether $(W_t)_{t\ge0}$ has a corresponding property.
\end{enumerate}
\end{remark}

\begin{remark}
Similarly as in Theorem \ref{FCLT-A} we may deduce functional central limit theorems for the
 powers $X_{t,k}^p$ ($p\in \mathbb N$) where these powers  are taken in all coordinates.
 The   most prominent examples appear for $p=2$. We have
\begin{equation}\label{power-limit-A}
\sqrt{k} ((\frac{X_{t,k}}{\sqrt k})^p -\phi^p(t,x)) \longrightarrow W_{t,p} \quad\quad \text{for}\quad k\to\infty 
\end{equation}
a.s.~locally uniformly in $t$  with  limit processes which solve
\begin{equation}\label{power-limit-A-dgl} dW_{t,p}=p \phi^{p-1}(t,x)dB_t+pA_p(t,x)W_{t,p}dt,\quad W_{0,p}=0
\end{equation}
with the matrices $A_p(t,x)\in \mathbb R^{N\times N}$ with
$$A_p(t,x)_{i,j}:= \frac{\phi_i^{p-1}(t,x)}{(\phi_i(t,x)-\phi_j(t,x))^2}\quad\quad \text{for}\quad i\ne j,$$
$$A_p(t,x)_{i,i}:=-\sum_{j\ne i} \frac{(p-2)\phi_i^{p-1}(t,x)-(p-1)\phi_i^{p-2}(t,x)\phi_j(t,x)}{(\phi_i(t,x)-\phi_j(t,x))^2}.$$
\end{remark}

\begin{remark}
Assume now that $N$ is odd, and that $(X_{t,k})_{t\geq0}$  starts at $\sqrt{k} cz$  in the interior of $C_N^A$.
 In order to study the $\frac{N-1}{2}+1$-th component, we notice that $z_{\frac{N-1}{2}+1}=0$ and that 
hence for $p\geq 2$ we have $W_t^{\frac{N-1}{2}+1}=0$ which implies a degenerate normal limiting distribution.
 This suggests that for the $\frac{N-1}{2}+1$-th component we have a faster rate of convergence than $\sqrt{k}$.
 Indeed from Theorem \ref{FCLT-A} we see that as $k\to\infty$
$$ X_{t,k}^{\frac{N-1}{2}+1}\stackrel{d}{\to} {\cal N}$$
for some normal random variable $\cal N$ and 
 hence $ (X_{t,k}^{\frac{N-1}{2}+1})^p\stackrel{d}{\to} {\cal N}^p$.
\end{remark}

We finally calculate the covariance matrix of $W_t$ for 
 the special solution $\phi$ given by (\ref{special-solution}).
For this we introduce the matrix $A\in \mathbb R^{N\times N}$ with
\begin{equation}\label{matrix-a-A}
 A_{i,i}:=-\sum_{j\ne i} \frac{1}{(z_i-z_j)^2},\quad A_{i,j}:= \frac{1}{(z_i-z_j)^2} \quad \text{for}\quad i\ne j
\end{equation}
with $z$ as in (\ref{char-zero-A}). 
Moreover, let $E$ be the N-dimensional unit matrix.
It is shown in  \cite{AV2} that the matrix  $E-A$  has the
 eigenvalues $1,2,\ldots,N$.

\begin{lemma}\label{cov-matrix-W}
Assume that $(X_{t,k})_{t\geq 0}$ starts in  the interior of $C_N^A$ in  $\sqrt{k}\cdot cz+y$ with
 $y\in\mathbb R$, $z$ as in (\ref{char-zero-A}) and $c>0$. 
Then  the covariance matrices $\Sigma_t\in \mathbb R^{N\times N}$ for $t>0$ of the limit process $(W_t)_{t\geq 0}$ are given by
$$ \Sigma_t=(t+\frac{c^2}{2})(E-A)^{-1}(E-e^{(\ln{\frac{c^2}{2t+c^2}})(E-A)}).$$
with eigenvalues 
$ \lambda_k^A(t,c)=\frac{1}{2k}\frac{(2t+c^2)^k-c^{2k}}{(2t+c^2)^{k-1}}$ ($k=1,\ldots,N$), where $\lambda_1^A(t,c)=t$.
\end{lemma}

\begin{proof}
For the special case $\phi(s,c z)= \sqrt{2s+c^2} z$,
 the matrix function $A(s,cz)$ satisfies $A(s,cz)=\frac{1}{2s+c^2}A$. Hence,
$$W_t=e^{(\ln(2t+c^2)-\ln c^2)A/2}\int_0^te^{(-\ln(2s+c^2)+\ln c^2)A/2}dB_s   \quad(t\ge0).$$
Since $A$ is real and symmetric with eigenvalues $1,\ldots,n$, we have
 $A=UDU^t$ with an orthogonal matrix $U$ and  the diagonal matrix
 $$D=diag(d_{1}, \ldots d_N):=diag(0,-1,-2,\ldots,-N+1).$$ 
Hence,
$$W_t=U\int_0^t diag\Biggl(\Bigl(\frac{2t+c^2}{2s+c^2}\Bigr)^\frac{d_{1}}{2}, \ldots, 
\Bigl(\frac{2t+c^2}{2s+c^2}\Bigr)^\frac{d_{N}}{2}\Biggr)d\tilde B_sU^t$$
with the rotated Brownian motion $(\tilde B_t:= U^tB_tU)_{t\ge0}$.
 This, the It\^{o}-isometry, and  $d_i/2\ne 1$ for all $i$ yield
\begin{eqnarray*}
\Sigma_t&=&U\cdot\int_0^t diag\Biggl(\Bigl(\frac{2t+c^2}{2s+c^2}\Bigr)^\frac{d_{1}}{2} , \ldots, 
\Bigl(\frac{2t+c^2}{2s+c^2}\Bigr)^\frac{d_{N}}{2}\Biggr)^2ds \cdot   U^t\\
&=&U\cdot\int_0^t
 diag\Biggl(\Bigl(\frac{2t+c^2}{2s+c^2}\Bigr)^{d_{1}} , \ldots,\Bigl(\frac{2t+c^2}{2s+c^2}\Bigr)^{d_{N}}\Biggr)
ds\cdot U^t\\
&=& U\cdot  diag\Biggl(\frac{1}{2(1-d_{1})}(2t+c^2-c^{2(1-d_1)}(2t+c^2)^{d_1}),  \ldots,\\
&&\quad\quad\quad\quad\quad\frac{1}{2(1-d_{N})}(2t+c^2-c^{2(1-d_N)}(2t+c^2)^{d_N})\Biggr)\cdot U^t.
\end{eqnarray*}
As
$$(U(E-D)^{-1}U^t)^{-1}=U(E-D)U^t=E-UDU^t=E-A$$
and
$$\frac{1}{2(1-d_{i})}(2t+c^2-c^{2(1-d_i)}(2t+c^2)^{d_i})
=\frac{1}{1-d_i}(t+\frac{c^2}{2})(1-e^{(1-d_i)\ln\frac{c^2}{2t+c^2}}),
$$
we obtain by functional calculus that
\begin{eqnarray*}
\Sigma_t&=&(t+\frac{c^2}{2})(E-A)^{-1}
\Biggl(E-U\cdot  diag\Biggl(e^{(1-d_1)\ln\frac{c^2}{2t+c^2}},  \ldots,e^{(1-d_N)\ln\frac{c^2}{2t+c^2}}\Biggr)U^t\Biggr)\\
&=&(t+\frac{c^2}{2})(E-A)^{-1}(E-e^{(\ln{\frac{c^2}{2t+c^2}})(E-A)})
\end{eqnarray*}
which yields the desired form of the covariance matrix.
\end{proof}

\begin{remark}
\begin{enumerate}
\item{   The covariance matrix $(E-A)^{-1}$ above already appeared in \cite{V} 
for the case $X_{0,k}=0$  and in \cite{DE2} in the context of asymptotics 
for eigenvalues of Hermite ensembles. Note that though we assume $c>0$ we may formally set $c=0$ and obtain $\Sigma_t= t(E-A)^{-1}$ as in \cite{V}. 
Since $\lim_{t\to\infty}\Sigma_t/t= (E-A)^{-1}$, we obtain asymptotically in $t$ the same result
 as starting in zero independent of the actual starting point $cz$.}
\item{ As
$$\Sigma_t=(t+\frac{c^2}{2})\biggl[-\ln\frac{c^2}{2t+c^2} E-\frac{(\ln\frac{c^2}{2t+c^2})^2}{2} (E-A)-\cdots\biggr],$$ 
 the diagonal elements have a behavior in $t$ which is different from the remaining entries of $\Sigma_t$.}
\item{  Lemma \ref{cov-matrix-W} holds also for $t=0$ with
 $\Sigma_0=0$.}
%\item{The methods of the proof of Lemma \ref{cov-matrix-W} also 
%lead to explicit formulas for the covariances of arbitrary components
% at different  times for  $(W_t)_{t\ge0}$.}
\end{enumerate}
\end{remark}
%\begin{remark}
%The eigenvalues of  $A(t,x)$ can be calculated in principle, since the characteristic polynomials are again symmetric polynomials o%f the entries, which results in quotients of symmetric polynomials in $\phi$.
%\end{remark}

\section{The root system  $B_{N}$}

We now turn to Bessel processes 
for the root systems  $B_N$ with multiplicities $k=(k_1,k_2)=(\nu\cdot\beta, \beta)$
with $\nu>0$ fixed and $\beta\to\infty$. 
We first  recapitulate some facts from
 \cite{AKM2} and \cite{AV}.
The SDE (\ref{SDE-general})  of type $B_N$ here reads as
\begin{equation}\label{SDE-B}
 dX_{t,k}^i = dB_t^i+ \beta \sum_{j\ne i} \Bigl(\frac{1}{X_{t,k}^i-X_{t,k}^j}+ \frac{1}{X_{t,k}^i+X_{t,k}^j}\Bigr)dt
+ \frac{\nu\cdot\beta}{X_{t,k}^i}dt\end{equation}
for $i=1,\ldots,N$
with an $N$-dimensional Brownian motion $(B_t^1,\ldots,B_t^N)_{t\ge0}$.
The  renormalized processes $(\tilde X_{t,k}:=X_{t,k}/\sqrt \beta)_{t\ge0}$ satisfy
\begin{equation}\label{SDE-B-normalized}
d\tilde X_{t,k}^i =\frac{1}{\sqrt \beta}dB_t^i + \sum_{j\ne i} 
 \Bigl( \frac{1}{\tilde X_{t,k}^i-\tilde X_{t,k}^j}
+ \frac{1}{\tilde X_{t,k}^i+\tilde X_{t,k}^j}\Bigr)dt +\frac{\nu}{\tilde X_{t,k}^i}dt
\end{equation}
for  $i=1,\ldots,N$. These processes 
 are  related   with the deterministic limit case $\beta=\infty$; see Lemma 3.1 of \cite{AV}. By \cite{AV} we know that for $\nu>0$,  for each starting point $x_0$ in the interior of $C_N^B$, the ODE
\begin{equation}\label{deterministic-boundary-B}
\frac{dx}{dt}(t) =H(x(t))
\end{equation}
with 
$$H(x):=
\left(\begin{matrix}\sum_{j\ne1}\Bigl( \frac{1}{x_1-x_j}+  
 \frac{1}{x_1+x_j}\Bigr)+\frac{\nu}{x_1}\\ \vdots\\ \sum_{j\ne N}
 \Bigl(\frac{1}{x_N-x_j}+   \frac{1}{x_N+x_j}\Bigr)+\frac{\nu}{x_N}\end{matrix}\right)$$
 has 
 a unique solution  $\phi(t,x_0)$ in the interior of $C_N^B$ for $t>0$.
Furthermore, we obtain the special solution
\begin{equation}\label{special-solution-B1}
\phi(t,c\cdot y)= \sqrt{t+c^2}\cdot y
\end{equation} 
for the special starting points $c y$, with $c>0$ and 
$y\in C_N^B$ given by the vector with
\begin{equation}\label{y-max-B1}
(y_1^2, \ldots, y_N^2)=2(z_1^{(\nu-1)},\ldots, z_N^{(\nu-1)}).
\end{equation}
Here $z_1^{(\nu-1)},\ldots,z_N^{(\nu-1)}$ denote the ordered zeros of the Laguerre polynomials $L_N^{(\nu-1)}$.
  
 For this we recapitulate that for $\alpha>0$
the Laguerre polynomials  $(L_n^{(\alpha)})_{n\ge0}$ are orthogonal w.r.t.~the density $e^{-x}\cdot x^{\alpha}$.

As in  Section 2,  the solutions $\phi$  in (\ref{special-solution-B1}) are attracting in some way, and
 all solutions of the ODEs (\ref{deterministic-boundary-B}) can be determined explicitly after
 some transformation of coordinates; see \cite{VW}.
Moreover, similar to the preceding section,
\begin{align}\label{derivative-2-b}
\frac{d}{dt}\|\phi(t,x)\|^2 &= 2\sum_{i=1}^N \phi_i(t,x)\cdot \dot \phi_i(t,x) \notag\\
 &= 2\sum_{i,j=1,\ldots,N, i\ne j} \Bigl( \frac{\phi_i(t,x)}{\phi_i(t,x)-\phi_j(t,x)}+
\frac{\phi_i(t,x)}{\phi_i(t,x)+\phi_j(t,x)} \Bigr) + 2\nu N \notag\\
 &=2 N(N+\nu-1). 
\end{align}
As $\|\phi(0,x)\|^2 =\|x\|^2$, we obtain for $t\ge0$ that
\begin{equation}\label{ODE-norm-b}
\|\phi(t,x)\|^2 =2N(N+\nu-1)t+\|x\|^2.
\end{equation}

The
 solutions $\phi$ appear in 
the  following strong limit law from  \cite{AV}.

\begin{theorem}\label{SLLN-B} 
Let $\nu>0$. Let $x$ be a point in the interior of $C_N^B$, and  $y\in \mathbb R^N$.
Let $\beta_0\ge 1/2$ with $\sqrt \beta \cdot x+y$ in the interior of $C_N^B$ for $\beta\ge \beta_0$.
For  $\beta\ge \beta_0$, consider the Bessel processes $(X_{t,k})_{t\ge0}$ of type B
  with $k=(k_1,k_2)=(\beta\cdot\nu ,\beta)$, which start in $\sqrt \beta\cdot x+y$.
Then, for all $t>0$,
$$\sup_{0\le s\le t, \beta\ge\beta_0}\|X_{s,k}- \sqrt\beta \phi(s,x) \|<\infty \quad\quad \text{a.s..}$$
In particular,
$\lim_{\beta\to\infty}X_{t,(\nu\cdot \beta,\beta)}/\sqrt \beta= \phi(t,x)$ 
locally uniformly in $t$ a.s..
\end{theorem}

We now turn to an associated functional central limit theorem.
We fix  $x$ 
 in the interior of $C_N^B$  and consider the associated
$\phi(t,x)$. We also introduce the process 
$(W_t)_{t\ge0}$ on $\mathbb R^N$ as  unique  solution of the inhomogeneous linear SDE
 \begin{align}\label{SDE-W1-B}
dW_t^i =dB_t^i+ &\sum_{j\ne i} \Biggl(\frac{W_t^j-W_t^i}{(\phi_i(t,x)-\phi_j(t,x))^2}-
\frac{W_t^j+W_t^i}{(\phi_i(t,x)+\phi_j(t,x))^2}\Biggr)dt\notag\\
& -\frac{\nu\cdot W_t^i}{ \phi_i(t,x)^2}dt
\end{align}
for $i=1,\ldots,N$ 
with initial condition $W_0=0$.
 Notice that  all denominators are $\ne0$ for $t>0$.
(\ref{SDE-W1-B}) may be written in matrix notation as 
\begin{equation}\label{SDE-W2-B}
dW_t=dB_t+A_\nu(t,x)W_tdt 
\end{equation}
with the matrices $A_\nu(t,x)\in \mathbb R^{N\times N}$ with
 \begin{align}\label{SDE-W3-B}
A_\nu(t,x)_{i,j}&:= \frac{1}{(\phi_i(t,x)-\phi_j(t,x))^2}-\frac{1}{(\phi_i(t,x)+\phi_j(t,x))^2} \quad(i\ne j),\notag\\
 A_\nu(t,x)_{i,i}&:=\sum_{j\ne i}\Biggl( \frac{-1}{(\phi_i(t,x)-\phi_j(t,x))^2}-\frac{1}{(\phi_i(t,x)+\phi_j(t,x))^2}\Biggr)
-\frac{\nu}{ \phi_i(t,x)^2}
\end{align}
for $i,j=1,\ldots, N$. The process $(W_t)_{t\ge0}$ is given  by
\begin{equation}\label{SDE-W4-B} 
W_t=e^{\int_0^t A_\nu(s,x)ds}\int_0^te^{-\int_0^s A_\nu(u,x)du}d B_s \quad(t\ge0).
\end{equation}
This process is obviously Gaussian; we describe 
it more closely below. 
%It is related to the Bessel processes $( X_{t,k})_{t\ge0}$ as follows:

\begin{theorem}\label{FCLT-B}  Let $\nu>0$. Let $x$ be a point in the interior of $C_N^B$ and let $y\in \mathbb R^N$.  
Let $\beta_0\ge 1/2$ such that $\sqrt \beta \cdot x+y$ is in the interior of $C_N^A$ for $\beta\ge \beta_0$.

For  $\beta\ge \beta_0$ consider the Bessel processes $(X_{t,k})_{t\ge0}$ of type $B_{N}$ with $k=( \nu\beta,\beta)$
starting  at
$\sqrt \beta\cdot x+y$.
Then, for all $t>0$,
\begin{equation}\label{CLT-diff-B}
\sup_{0\le s\le t, \beta\geq \beta_0}\sqrt \beta \cdot \|X_{s,k}-\sqrt \beta \phi(s,x) -W_s\|<\infty  \quad\quad a.s.,
\end{equation}
i.e., $\lim_{\beta\to\infty}X_{s,k}-\sqrt \beta \phi(s,x)= W_t$
locally uniformly in $t$ a.s.~with  rate $O(1/\sqrt \beta)$.
\end{theorem}

\begin{proof}
For $\beta\ge \beta_0$ consider the processes 
$$(R_{t,\beta}:=X_{t,(\nu\beta,\beta)}- \sqrt \beta \phi(t,x) -W_t)_{t\ge0}$$
 on $\mathbb R^N$ with
$R_{0,\beta}=0$. Then  by  (\ref{SDE-W1-B}), (\ref{SDE-B}) and the ODE (\ref{deterministic-boundary-B}), $R_{t,\beta}^i$ is eqaul to 
\begin{align}
 \beta \int_0^t \Biggl[&\sum_{j\ne i}\Biggl( \frac{1}{X_{s,k}^i-X_{s,k}^j}-\frac{1}{\sqrt \beta (\phi_i(s,x)-\phi_j(s,x))}
-  \frac{W_s^j-W_s^i}{\bigl(\sqrt \beta (\phi_i(s,x)-\phi_j(s,x))\bigr)^2}\Biggr)\notag\\
+&\sum_{j\ne i}\Biggl( \frac{1}{X_{s,k}^i+X_{s,k}^j}-\frac{1}{\sqrt \beta (\phi_i(s,x)+\phi_j(s,x))}
+  \frac{W_s^j+W_s^i}{\bigl(\sqrt \beta (\phi_i(s,x)+\phi_j(s,x))\bigr)^2}\Biggr)\notag\\
+&\nu\Biggl( \frac{1}{X_{s,k}^i} - \frac{1}{\sqrt \beta \phi_i(s,x)}+ \frac{W_s^i}{\bigl(\sqrt \beta \phi_i(s,x)\bigr)^2}
\Biggr)\Biggr]ds.
\end{align}
The Taylor expansion for  $1/x$ with Lagrange remainder at $x_0\ne 0$ yields
$$ 1/x = 1/x_0- (x-x_0)/(x_0^2) + (x-x_0)^2/\tilde x^3$$
with some $\tilde x$ between $x\ne0$ and $x_0\ne0$ which have the same signs. Taking 
\begin{align}x&=X_{s,k}^i \pm X_{s,k}^j ,\quad\quad x_0=\sqrt \beta (\phi_i(s,x)\pm \phi_j(s,x)) \quad
\text{and}\notag\\
x&=X_{s,k}^i ,  \quad\quad \quad\quad \quad x_0=\sqrt \beta \phi_i(s,x),
\notag
\end{align}
we get
\begin{align}
R_{t,\beta}^i =& -\int_0^t\Biggl(\sum_{j\ne i}
\frac{(X_{s,k}^i-\sqrt \beta \phi_i(s,x)-W_s^i)-(X_{s,k}^j-\sqrt \beta \phi_j(s,x)-W_s^j)}{ (\phi_i(s,x)-\phi_j(s,x))^2}
\notag\\
&\quad 
+\sum_{j\ne i}
\frac{(X_{s,k}^i-\sqrt \beta \phi_i(s,x)-W_s^i)+(X_{s,k}^j-\sqrt \beta \phi_j(s,x)-W_s^j)}{ (\phi_i(s,x)+\phi_j(s,x))^2}
\notag\\
&\quad+ \frac{X_{s,k}^i-\sqrt \beta \phi_i(s,x)-W_s^i}{ \phi_i(s,x)^2}\> +\> 
+H_{s,\beta}^i\Biggr)ds
\notag\\
&= -\int_0^t\Biggl(\sum_{j\ne i}\frac{R_{s,\beta}^i-R_{s,\beta}^j}{ (\phi_i(s,x)-\phi_j(s,x))^2}
+\sum_{j\ne i}\frac{R_{s,k}^i+R_{s,k}^j}{ (\phi_i(s,x)+\phi_j(s,x))^2}
\notag\\
&\quad 
+\nu\frac{R_{s,\beta}^i}{ \phi_i(s,x)^2}
+H_{s,\beta}^i\Biggr)ds
\notag
\end{align}
with the error terms
\begin{align}
H_{s,\beta}^i=&\beta\Biggl[\sum_{j\ne i}
\frac{\Bigl((X_{s,k}^i-\sqrt \beta \phi_i(s,x))-(X_{s,k}^j-\sqrt \beta \phi_j(s,x))\Bigr)^2}{
\Bigl(\sqrt \beta (\phi_i(s,x)-\phi_j(s,x))+D_{i,j}^-(s)\Bigr)^3}\notag\\
&\quad +
\sum_{j\ne i}
\frac{\Bigl((X_{s,k}^i-\sqrt \beta \phi_i(s,x))+(X_{s,k}^j-\sqrt \beta \phi_j(s,x))\Bigr)^2}{
\Bigl(\sqrt \beta (\phi_i(s,x)+\phi_j(s,x))+D_{i,j}^+(s)\Bigr)^3}
\notag\\
&\quad + \nu
\frac{\Bigl(X_{s,k}^j-\sqrt \beta \phi_j(s,x))\Bigr)^2}{
\Bigl(\sqrt \beta \phi_i(s,x)+D_{i}(s)\Bigr)^3}
\biggr]
\notag
\end{align}
where, by the Lagrange remainders in the 3 Taylor expansions above,
$$|D_{i,j}^\pm(s)|\le |(X_{s,k}^i-\sqrt \beta \phi_i(s,x))\pm(X_{s,k}^j-\sqrt \beta \phi_j(s,x))|$$
and
$$|D_{i}(s)|\le |X_{s,k}^i-\sqrt \beta \phi_i(s,x)|.$$
By Theorem \ref{SLLN-B}, the terms $|D_{i,j}^\pm|,|D_{i}|$ 
can be bounded by some a.s.~finite random variable $D$ independent
 of $i,j$, the sign, $s\in[0,t]$, and $\beta\ge \beta_0$ where $D$ depends on $x,y,t$.
 Therefore, for all $i=1,\ldots,N$,
$$|H_{s,\beta}^i|\le \frac{1}{\sqrt \beta}H \quad\quad\text{for}\quad \beta\ge \beta_0, \> s\in[0,t]$$
 with some  a.s.~finite random variable $H$. In summary,
$$R_{t,\beta}=-\int_0^t( A_\nu(s,x)R_{s,\beta} + H_{s,\beta})\> ds, \quad R_{0,k}=0$$
and thus, for suitable norms and all $u\in[0,t]$, 
$$\|R_{u,\beta}\|\le A\int_0^u\|R_{s,\beta}\|\> ds + \frac{t\cdot \|H\|}{\sqrt \beta}$$
with $A:= \sup_{s\in [0,t]} \|A_\nu(s,x)\|<\infty$. Hence, by the classical lemma of Gronwall,
$$\|R_{u,\beta}\|\le\frac{t\|H\|}{\sqrt \beta}e^{tA}$$
for all  $u\in[0,t]$. This yields the claim.
\end{proof}

%\begin{remark}\label{remark-verhalten-b}
%The Bessel processes $(X_{t,k})_{t\ge0}$ of type B have the same space-time scaling as Brownian motions
%and the Bessel processes of type A in Remark \ref{remark-verhalten-a}(1), i.e.,  for  $r>0$, 
% $(\frac{1}{r} X_{r^2t, k})_{t\ge0}$ is also a  Bessel process of type B with the same $k$
%with modified starting points.  The corresponding relations for $\phi$ and  $A$ are
%$$\phi(r^2t, rx)=r\cdot \phi(t,x), \quad A(r^2t, rx)=\frac{1}{r^2}A(t,x)\quad \text{for}\quad
%r>0, \> t\ge0.$$
%Moreover, for the solution $(W_t)_{t\ge0}$ of (\ref{SDE-W1-A}), 
%$ (\frac{1}{r} W_{r^2t})_{t\ge0}$ is also a process of this type where $x$ has to be replaced by $rx$ in
% Eqs.~(\ref{SDE-W1-B})--(\ref{SDE-W3-B}). 
%\end{remark}

We next calculate the covariance matrix of $W_t$ for 
 the special solution $\phi$ of (\ref{special-solution-B1}).
For this we introduce the matrices $A_\nu=(A_{\nu,i,j})_{i,j}\in \mathbb R^{N\times N}$ with
\begin{equation}\label{matrix-a-B}
A_{\nu,i,j}:= \frac{1}{(y_i-y_j)^2}-\frac{1}{(y_i+y_j)^2}  , \quad
 A_{\nu,i,i}:=\sum_{j\ne i} \Bigl( \frac{-1}{(y_i-y_j)^2}-\frac{1}{(y_i+y_j)^2} \Bigr)-\frac{\nu}{y_i^2}
\end{equation}
for $i,j=1,\ldots, N$, $i\ne j$ and $y$ as in (\ref{y-max-B1}). 
By  \cite{AV2},   $E-2A_\nu$  has the
 eigenvalues $ 2,4,\ldots,2N$
independent of $\nu$.
With these notations we have:

\begin{lemma}\label{cov-matrix-W-B}
Let $\nu>0$.
Assume that the Bessel process $(X_{t,k})_{t\geq 0}$ of type B  with $k=(\nu\beta, \beta)$
starts in the point  $\sqrt{\beta}\cdot c y+w$ in the interior of $C_N^B$ with
 $w\in\mathbb R$, $y$ as in (\ref{y-max-B1}), and $c>0$. 
Then,  the covariance matrices $\Sigma_{\nu,t}\in \mathbb R^{N\times N}$ for $t>0$ of the limit 
Gaussian process $(W_t)_{t\geq 0}$ are given by
$$ \Sigma_{\nu,t}=(t+{c^2})(E-2A_\nu)^{-1}(E-e^{\ln{\frac{c^2}{t+c^2}}(E-2A_\nu)})$$
with eigenvalues
 $\lambda_k^B(t,c)=\frac{1}{2k}\frac{(t+c^2)^{2k}-c^{4k}}{(t+c^2)^{2k-1}}$ $(k=1,\ldots,N$).
\end{lemma}

\begin{proof}
For the special case $\phi(s,c y)= \sqrt{s+c^2} y$ with $c>0$ and the vector $y$  in (\ref{y-max-B1}),
 the matrix function $A_\nu(s,cz)$ has the  form $A_\nu(s,cy)=\frac{1}{s+c^2}A_\nu$. Hence,
$$W_t=e^{(\ln(t+c^2)-\ln c^2)A_\nu}\int_0^te^{(-\ln(s+c^2)+\ln c^2)A_\nu}dB_s   \quad(t\ge0).$$
Since $A_\nu$ is real and symmetric with eigenvalues $2,4,\ldots,2N$, we may write  $A_\nu$ as
 $A_\nu=UDU^t$ with an orthogonal matrix $U$ and the diagonal matrix
 $$D=diag(d_{1}, \ldots d_N):=diag(-1/2,-3/2,\ldots,(-2N+1)/2).$$ 
Hence,
$$W_t=U\int_0^t diag\Biggl(\Bigl(\frac{t+c^2}{s+c^2}\Bigr)^{d_{1}}, \ldots, 
\Bigl(\frac{t+c^2}{s+c^2}\Bigr)^{d_{N}}\Biggr)d\tilde B_s\cdot U^t$$
with the rotated Brownian motion $(\tilde B_t:= U^tB_tU)_{t\ge0}$.
 Now similar arguments as in the proof of Lemma \ref{cov-matrix-W} yield the desired form of the covariance matrix.
\end{proof}

\begin{remark}
The eigenvalues of $\Sigma_t$ in the cases $A_{2N-1}$ and $B_N$ are related by
$$ \lambda_i^B(t,c)=\lambda_{2i}^A(t,c\cdot\sqrt 2)  \quad\quad(i=1,\ldots,N)$$
independent of $\nu$.
This seems to be connected in some way with the formula $H_{2N}(x)=const.(N)\cdot L_N^{(-1/2)}(x^2)$
for $\nu=1/2$. 
\end{remark}

All preceding results hold for  $\nu>0$. 
We show in Section 4 that most results are also valid for  $\nu=0$ with some modifications.
%The case $\nu=0$ is closely related to the root systems of type $D$.

\section{The root system $D_N$}\label{D-N}

We now briefly study limit theorems for Bessel processes of type $D_N$. We recapitulate that the associated closed Weyl chamber is
$$C_N^D=\{x\in\mathbb R^N: \quad x_1\ge \ldots\ge x_{N-1}\ge |x_N|\},$$
i.e., $C_N^D$ is  a doubling of $C_N^B$ w.r.t.~the last coordinate. We have a one-dimensional multiplicity $k\ge0$, and
the SDE (\ref{SDE-general}) for the  processes $(X_{t,k})_{t\ge0}$ of type $D$ is
\begin{equation}\label{SDE-D}
 dX_{t,k}^i = dB_t^i+ k\sum_{j\ne i} \Bigl(\frac{1}{X_{t,k}^i-X_{t,k}^j}+ \frac{1}{X_{t,k}^i+X_{t,k}^j}\Bigr)dt\end{equation}
for $i=1,\ldots,N$
with an $N$-dimensional Brownian motion $(B_t^1,\ldots,B_t^N)_{t\ge0}$.
The  renormalized processes $(\tilde X_{t,k}:=X_{t,k}/\sqrt k)_{t\ge0}$ satisfy
\begin{equation}\label{SDE-D-normalized}
d\tilde X_{t,k}^i =\frac{1}{\sqrt k}dB_t^i + \sum_{j\ne i} 
 \Bigl( \frac{1}{\tilde X_{t,k}^i-\tilde X_{t,k}^j}
+ \frac{1}{\tilde X_{t,k}^i+\tilde X_{t,k}^j}\Bigr)dt \quad (i=1,\ldots,N).
\end{equation}
For  $k=\infty$ we have
 from Lemma 4.1 of \cite{AV} that 
for
  each starting point $x_0$ in the interior of $C_N^D$, 
the ODE 
\begin{equation}\label{deterministic-boundary-D}
\frac{dx}{dt}(t) =H(x(t))
\end{equation}
with
$$H( x):=
\left(\begin{matrix}\sum_{j\ne1}\Bigl( \frac{1}{x_1-x_j}+   \frac{1}{x_1+x_j}\Bigr)\\
\vdots\\
\sum_{j\ne N} \Bigl(\frac{1}{x_N-x_j}+   \frac{1}{x_N+x_j}\Bigr)
\end{matrix}\right)$$
 has a unique 
 solution $\phi(t,x_0)$ in the interior of $C_N^D$ for  $t\ge0$.
Again, as in Section 4 of \cite{AV}, a special solution associated to the zeros of Laguerre polynomials may be derived.
Using the  representation
$$L_N^{(\alpha)}(x):=\sum_{k=0}^N { N+\alpha\choose N-k}\frac{(-x)^k}{k!} \quad\quad(\alpha\in\mathbb R, \> N\in\mathbb N)$$
of the Laguerre polynomials according to (5.1.6) of Szeg\"o \cite{S}, we  form the polynomial 
 $L_N^{(-1)}$ of order $N\ge1$
where,
 by (5.2.1) of \cite{S}, 
\begin{equation}\label{laguerre-1}
L_N^{(-1)}(x)=-\frac{x}{N}L_{N-1}^{(1)}(x).
\end{equation}
We now denote the $N$ zeros of $L_N^{(-1)}$ by  $z_1>\ldots>z_N=0$ and define
\begin{equation}\label{y-max-D}
2\cdot (z_1^{(1)},\ldots, z_{N-1}^{(1)},0)= (r_1^2,\ldots,r_N^2).
\end{equation}
 Then we obtain for the starting points $c r$ with $c>0$, the  particular solutions
\begin{equation}\label{special-solution-D}
\phi(t,c\cdot r)= \sqrt{t+c^2}\cdot r
\end{equation}
 of the ODE (\ref{deterministic-boundary-D}); cf.~\cite{AV}.

Again, the solutions (\ref{special-solution-D}) are attracting in some way. Moreover, 
as in the preceding sections, we have
\begin{equation}\label{ODE-norm-d}
\|\phi(t,x)\|^2 =2N(N-1)t+\|x\|^2  \quad\quad(t\ge0).
\end{equation}

Beside the particular solutions given by (\ref{special-solution-D}) 
we  have the following observation which fits with Eq.~(\ref{laguerre-1}).

\begin{lemma}\label{special-lemma-d}
Let $x$ be a point in the interior of $C_N^D$ with $x_N=0$. Then the associated solution 
of the ODE (\ref{deterministic-boundary-D}) satisfies $\phi(t,x)_N=0$ for all $t$,
 and the first $N-1$ components
 $(\phi(t,x)_1,\ldots,\phi(t,x)_{N-1})$ solve the ODE of the B-case in (\ref{deterministic-boundary-B})
 with dimension $N-1$ and $\nu=2$.
\end{lemma}

\begin{proof}
If $x_N=0$, then by the ODE  in  (\ref{deterministic-boundary-D}), $\frac{d}{dt}\phi(t,x)_N=0$.
\end{proof}

We next recapitulate the following strong limit law; see Theorem 5.5 of \cite{AV}:

\begin{theorem}\label{SLLN-D}
 Let $x$ be a point in the interior of $C_N^D$, and  $y\in \mathbb R^N$. 
Let $k\ge 1/2$ with $\sqrt k \cdot x+y$ in the interior of $C_N^B$ for $k\ge k_0$.
 For  $k\ge k_0$, consider the Bessel processes $(X_{t,k})_{t\ge0}$ of type $D_N$ starting in $\sqrt k\cdot x+y$.
Then, for all $t>0$,
$$\sup_{0\le s\le t, k\ge k_0}\|X_{s,k}- \sqrt k \phi(s,x) \|<\infty \quad\quad a.s..$$
In particular,
$X_{t,k}/\sqrt k\to \phi(t,x)$ for $k\to\infty$
locally uniformly in $t$ a.s..
\end{theorem}

We now turn to an associated functional CLT.
We fix some $x$ 
 in the interior of $C_N^D$  and consider the associated
solution $\phi(t,x)$. We also introduce an $N$-dimensional process 
$(W_t)_{t\ge0}$ as the unique  solution of the inhomogeneous linear SDE
 \begin{equation}\label{SDE-W1-D}
dW_t^i =dB_t^i+ \sum_{j\ne i} \Biggl(\frac{W_t^j-W_t^i}{(\phi_i(t,x)-\phi_j(t,x))^2}-
\frac{W_t^j+W_t^i}{(\phi_i(t,x)+\phi_j(t,x))^2}\Biggr)dt
\end{equation}
for $i=1,\ldots,N$ 
with initial condition $W_0=0$. 
(\ref{SDE-W1-D}) may be written  as 
\begin{equation}\label{SDE-W2-D}
dW_t=dB_t+A(t,x)W_tdt 
\end{equation}
with the matrix $A(t,x)\in \mathbb R^{N\times N}$ with
 \begin{align}\label{SDE-W3-D}
A(t,x)_{i,j}&:= \frac{1}{(\phi_i(t,x)-\phi_j(t,x))^2}-\frac{1}{(\phi_i(t,x)+\phi_j(t,x))^2} \quad(i\ne j),\notag\\
 A(t,x)_{i,i}&:=\sum_{j\ne i}\Biggl( \frac{-1}{(\phi_i(t,x)-\phi_j(t,x))^2}-\frac{1}{(\phi_i(t,x)+\phi_j(t,x))^2}\Biggr)
\end{align}
for $i,j=1,\ldots, N$. The process $(W_t)_{t\ge0}$ is Gaussian and given  by
\begin{equation}\label{SDE-W4-D} 
W_t=e^{\int_0^t A(s,x)ds}\int_0^te^{-\int_0^s A(u,x)du}d B_s \quad(t\ge0).
\end{equation}
 It is related to the Bessel processes $( X_{t,k})_{t\ge0}$ of type $D$ by the following result.
As the proof is completely analogous to that of Theorem \ref{FCLT-B}, we omit the proof.

\begin{theorem}\label{FCLT-D} Let $x$ be a point in the interior of $C_N^D$ and  $y\in \mathbb R^N$.  
Let $k_0\ge 1/2$ such that $\sqrt k \cdot x+y$ is in the interior of $C_N^D$ for $k\ge k_0$.
For  $k\ge k_0$ consider the Bessel processes $(X_{t,k})_{t\ge0}$ 
starting  at
$\sqrt k\cdot x+y$.
Then, for all $t>0$,
\begin{equation}\label{CLT-diff-D}
\sup_{0\le s\le t, k\geq k_0}\sqrt k \cdot \|X_{s,k}-\sqrt k \phi(s,x) -W_s\|<\infty  \quad\quad a.s.,
\end{equation}
i.e., $X_{s,k}-\sqrt k \phi(s,x) \to W_s$
locally uniformly in $s$ a.s.~with  rate $O(1/\sqrt k)$.
\end{theorem}

\begin{remark}\label{d-fall-ann-null}
Consider the  Bessel processes $(X_{t,k})_{t\ge0}$ of Theorem  \ref{FCLT-D} which start in $\sqrt k\cdot x$ for  $x$ 
 in the interior of $C_N^D$ with $x_N=0$. Then, by Lemma \ref{special-lemma-d}, $\phi(t,x)_N=0$ for all $t\ge0$, and
 the matrix function $A$ from (\ref{SDE-W3-D}) satisfies
$A(t,x)_{N,N}=0$.
\end{remark}

We next calculate the covariance matrix of $W_t$ for 
 the special solution $\phi$ in (\ref{special-solution-D}).
For this we introduce the matrix $A\in \mathbb R^{N\times N}$ with
\begin{equation}\label{matrix-a-DS}
A_{i,j}:= \frac{1}{(r_i-r_j)^2}-\frac{1}{(r_i+r_j)^2}  , \quad
 A_{i,i}:=\sum_{j\ne i} \Bigl( \frac{-1}{(r_i-r_j)^2}-\frac{1}{(r_i+r_j)^2} \Bigr)
\end{equation}
for $i,j=1,\ldots, N$, $i\ne j$ and the vector $r$ as in  (\ref{y-max-D}). 
By \cite{AV2}, $E-2A$  has the
 eigenvalues
$ 2,4,\ldots,2N$.
We obtain the following result. As its proof is again  analog to 
that of Lemma \ref{cov-matrix-W-B}, we skip the proof.

\begin{lemma}\label{cov-matrix-W-D}
Assume that the Bessel processes $(X_{t,k})_{t\geq 0}$ of type $D_N$
start in the points  $\sqrt{k}\cdot c r+w$ in the interior of $C_N^D$ with
 $w\in\mathbb R^N$, $r$ given in (\ref{y-max-D}), and $c>0$. 
Then,  the covariance matrices $\Sigma_{t}\in \mathbb R^{N\times N}$ for $t>0$ of the limit Gaussian process
 $(W_t)_{t\geq 0}$ are 
$$ \Sigma_{t}=(t+{c^2})(E-2A)^{-1}(E-e^{\ln{\frac{c^2}{t+c^2}}(E-2A)}).$$
\end{lemma}

\begin{remark}\label{independence}
Let $x$ be in the interior of $C_N^D$ with $x_N=0$. Then, by Lemma \ref{special-lemma-d}, 
 $A(t,x)_{N,j}=A(t,x)_{j,N}=0$ for $j\ne N$ and $t>0$ in (\ref{SDE-W3-D}). Hence, by
 Eq.~(\ref{SDE-W4-D}), that  the $N$-th component
 $(W_t^{(N)})_{t\geq 0}$ of $(W_t)_{t\geq 0}$ is independent from  $(W_t^{(1)}, \ldots,W_t^{(N-1)})_{t\geq 0}$.
This  appears in particular in the setting of Lemma \ref{cov-matrix-W-D}.
\end{remark}

\section{Further limit theorems for the case B}

 Section 4 is closely related to limit results for Bessel processes of 
type B for  $(k_1,k_2)=(0,k)$ for $k\to\infty$ which was excluded in Section 3.
To explain this, we recapitulate some facts from \cite{AV}, \cite{V}. 
Let $(X_{t,k}^D)_{t\ge0}$ be a Bessel process of type D with multiplicity $k\ge0$ on  $C_N^D$ with
starting  point $x$ in the interior of $C_N^D$. It follows from the generator of the associated semigroup
  that then the process
$(X_{t,k}^B)_{t\ge0}$ with 
$$X_{t,k}^{B,i}:= X_{t,k}^{D,i} \quad(i=1,\ldots,N-1), \quad X_{t,k}^{B,N}:= |X_{t,k}^{D,N} |$$
is a Bessel process of type B with the multiplicity $(k_1,k_2):=(0,k)$ and starting point 
$(x_1,\ldots,x_{N-1},|x_N|)\in C_N^B$ with $x_1>\ldots>x_{N-1}>|x_N|\ge0$. Notice that 
$(X_{t,k}^B)_{t\ge0}$ is a diffusion with reflecting boundary where
 the boundary parts with the $N$-th coordinate equal to zero are attained.

We now translate the results of Section 4. For this we consider the solutions $\phi(t,x)$ of the ODE (\ref{deterministic-boundary-D}) in the following two particular cases for  $t\ge0$:
\begin{enumerate}
\item[\rm{(1)}] If $x$ is in the interior of $C_N^B$, then  $\phi(t,x)$ is also in the interior  of $C_N^B$.
\item[\rm{(2)}] If  $x_1>\ldots>x_{N-1}>x_N=0$, then 
 $\phi(t,x)_1>\ldots>\phi(t,x)_{N-1}>\phi(t,x)_N=0$.
\end{enumerate}
Case (2) appears in particular for $\phi(t,x)=\sqrt{t+c^2}\cdot r$ for $c>0$ and the vector $r$ as in (\ref{special-solution-D}) with $r_N=0$.

Theorem \ref{SLLN-D} now reads as follows for the B-case  with $(k_1,k_2)=(0,k)$ for $k\to\infty$:

\begin{theorem}\label{SLLN-B2}
 Let $x$ be  as described above in (1) or (2). 
 For  $k\ge 1/2$, consider the Bessel processes $(X_{t,(0,k)})_{t\ge0}$ of type $B_N$ starting in $\sqrt k\cdot x$.
Then, for $t>0$,
$$\sup_{0\le s\le t, k\ge 1/2}\|X_{s,k}- \sqrt k \phi(s,x) \|<\infty \quad\quad a.s..$$
\end{theorem}

We next consider the Gaussian processes $(W_t)_{t\ge0}$ of Eq.~(\ref{SDE-W4-D}).  Theorem \ref{FCLT-D}
now leads to functional CLTs where the cases (1) and (2) have to be treated separately for geometric reasons.
For the case (1) we have the following result: 

\begin{theorem}\label{FCLT-B1} Let $x$ be a point in the interior of $C_N^B$.  
For  $k\ge 1/2$ consider  Bessel processes $(X_{t,(0,k)})_{t\ge0}$ of type B
starting  at
$\sqrt k\cdot x$.
Then, for all $t>0$,
\begin{equation}\label{CLT-diff-B-deg}
\sup_{0\le s\le t, k\geq k_0}\sqrt k \cdot \|X_{s,(0,k)}-\sqrt k \phi(s,x) -W_s\|<\infty  \quad\quad a.s..
\end{equation}
\end{theorem}

\begin{proof}
As  $x$ is  in the interior of $C_N^B$, we obtain that for each $t>0$ and almost all $\omega\in\Omega$,
the path  $(\sqrt k \phi(s,x)-W_s(\omega))_{s\in [0,t]}$ is arbitrarily far away from the boundary of  $C_N^B$ whenever
 $k$ is sufficiently large. This, the connection between the D- and B-case, and  Theorem \ref{FCLT-D} thus lead to the
theorem.
\end{proof}

Theorem \ref{FCLT-B1} corresponds to Theorem  \ref{FCLT-B} for  $\nu=0$.
We next turn to case (2):

\begin{theorem}\label{FCLT-B2} Let $x\in C_N^B$ with  $x_1>\ldots>x_{N-1}>x_N=0$. 
For  $k\ge 1/2$ consider  Bessel processes $(X_{t,(0,k)})_{t\ge0}$ of type B
starting  at
$\sqrt k\cdot x$.
Then, for the process $(\tilde W_t:=(W_t^{(1)},\ldots, W_t^{(N-1)},|W_t^{(N)}|))_{t\ge0}$, and all
 $t>0$,
\begin{equation}\label{CLT-diff-B1}
\sup_{0\le s\le t, k\geq k_0}\sqrt k \cdot \|X_{s,(0,k)}-\sqrt k \phi(s,x) -\tilde W_s\|<\infty  \quad\quad a.s..
\end{equation}
\end{theorem}

\begin{proof} This follows immediately from  Theorem \ref{FCLT-D},  the connection between the D- and B-case, and from
$||a|-|b||\le |a-b|$ for $a,b\in\mathbb R$.
\end{proof}

For  the process  $(\tilde W_t)_{t\ge0}$, the first $N-1$ components form a Gaussian process
 which is independent from $(|W_t^{(N)}|)_{t\ge0}$ by Remark \ref{independence}. The variables
$|W_t^{(N)}|$  are one-sided normal distributed. Therefore, Theorems \ref{FCLT-B1} and \ref{FCLT-B2}
lead to a discontinuity (or phase transition) in the limit depending on the starting points here.

\section{Extensions to multi-dimensional Bessel processes with an additional Ornstein-Uhlenbeck component}
In this section we  consider an extension of our previous models by adding an additional drift coefficient of the form $-\lambda x$, $\lambda\in\mathbb R$, i.e., a component as in a classical Ornstein-Uhlenbeck setting
$$ dY_{t,k}= dB_t +  (\frac{1}{2} (\nabla(\ln w_k))(Y_{t,k})-\lambda Y_{t,k}) \> dt.$$
 If $\lambda>0$, we obtain a mean reverting  ergodic process with speed of mean-reversion $\lambda$.
 For $\lambda\leq 0$ the process is non-ergodic. For $N=1$ and $\lambda>0$ the squared process is 
the well-known Cox-Ingersoll-Ross process from  mathematical finance.

We derive the results for the root system $A_{N-1}$ only as the same technique also holds for the other root systems.
 We consider processes $( Y_{t,k})_{t\ge0}$
of type $A_{N-1}$  as solutions of
\begin{equation}\label{SDE-A-mr}
 dY_{t,k}^i = dB_t^i+\left( k\sum_{j\ne i} \frac{1}{Y_{t,k}^i-Y_{t,k}^j}-\lambda  Y_{t,k}^i\right) dt \quad\quad(i=1,\ldots,N).
\end{equation}
with an $N$-dimensional Brownian motion $(B_t^1,\ldots,B_t^N)_{t\ge0}$.  It\^{o}'s formula and 
a time-change argument show that $Y$ is a space-time transformation 
of the original $X$ (with $\lambda=0$), namely 
$$Y_{t,k}=e^{-\lambda t}X_{\frac{e^{2\lambda t}-1}{2\lambda},k}.$$
 For a proof based on the generators cf. \cite{RV1}.
A similar relation  holds for the  associated ODEs.

\begin{lemma}\label{deterministic-dynamical-system-mr}
Let $\phi(t,x)$ be a solution of the dynamical system $\frac{dx}{dt}(t) =H(x(t))$
 with starting point $x$ in the interior of $C_N^A$ as in Section 2.  Then 
$$\phi(\frac{1-e^{-2\lambda t}}{2\lambda}, e^{-\lambda t}x)$$ solves the ODE
  $\frac{dx}{dt}(t) =H(x(t))-\lambda x(t)$ with starting point $x$. 
\end{lemma}

\begin{proof}
This follows from the space-time
 homogeneity in Remark \ref{remark-verhalten-a}.
\end{proof}

With the techniques in Theorem \ref{SLLN-A} and Theorem \ref{FCLT-A} we obtain a functional CLT 
 for $(Y_{t,k})_{t\geq 0}$:  
\begin{equation}\sqrt{k} \Bigl(\frac{Y_{t,k}}{\sqrt k} -\phi(\frac{1-e^{-2\lambda t}}{2\lambda}, e^{-\lambda t}x)\Bigr) \longrightarrow W_t
\end{equation}
for $k\to\infty$ locally uniformly in $t$ a.s.~with rate $O(1/\sqrt k)$, where 
\begin{equation}\label{SDE-W2-A-mr}
dW_t=dB_t+A^\lambda(t,x)W_tdt 
\end{equation}
with the matrices $A^\lambda(t,x)\in \mathbb R^{N\times N}$ with
$$A^\lambda(t,x)_{i,j}:= \frac{1}{(\phi_i(\frac{1-e^{-2\lambda t}}{2\lambda}, e^{-\lambda t}x)-\phi_j(\frac{1-e^{-2\lambda t}}{2\lambda}, e^{-\lambda t}x))^2},$$
 $$ A^\lambda(t,x)_{i,i}:=-\sum_{j\ne i} \frac{1}{(\phi_i(\frac{1-e^{-2\lambda t}}{2\lambda}, e^{-\lambda t}x)-\phi_j(\frac{1-e^{-2\lambda t}}{2\lambda}, e^{-\lambda t}x))^2}-\lambda$$
for $i,j=1,\ldots, N$, $i\ne j$ and with initial condition $W_0=0$. The process $(W_t)_{t\ge0}$ admits
the explicit representation 
\begin{equation}\label{SDE-W3-A-mr} 
W_t=e^{\int_0^t A^\lambda(s,x) ds}\int_0^te^{-\int_0^s A^\lambda(u,x) du}d B_s \quad(t\ge0).
\end{equation}
Note that due to the constant term in the diagonal of $A^\lambda(t,x)$ for $\lambda\neq 0$, 
we obtain a linear time-dependence in the exponential of the matrix exponential which dominates
 the long-term behaviour of the covariance matrix. In particular:

\begin{lemma}\label{cov-matrix-W-mr1}
Assume that $(Y_{t,k})_{t\geq 0}$ starts in  the interior of $C_N^A$ in  $\sqrt{k}\cdot cz+y$ with
 $y\in\mathbb R$, $z$ as in \ref{char-zero-A} and $c>0$. 
Then  the covariance matrices $\Sigma^\lambda_t\in \mathbb R^{N\times N}$ for $t>0$ of the limit process $(W_t)_{t\geq 0}$ are given by
$$ \Sigma^\lambda_t=\frac{1+e^{-2\lambda t}(\lambda c^2-1)}{2\lambda}(E-A)^{-1}(E-e^{\ln{\frac{\lambda c^2}{e^{2\lambda t}-1+\lambda c^2}}(E-A)}),$$
where $A$ is defined by (\ref{matrix-a-A}).
\end{lemma}

\begin{proof}
For the special starting points $cz$ we obtain the special solution 
$$
\phi(\frac{1-e^{-2\lambda s}}{2\lambda}, e^{-\lambda s}cz)= e^{-\lambda s}\sqrt{\frac{e^{2\lambda s}-1}{\lambda}+c^2} z.
$$
 Hence the matrix function $A^\lambda(s,cz)$ has the  simple form with the same time-dependence for each entry 
 $$A^\lambda(s,cz)=\frac{\lambda e^{2\lambda s}}{e^{2\lambda s}-1+\lambda c^2}A-diag(\lambda s, \cdots, \lambda s),$$
where $A$ is given by (\ref{matrix-a-A}).  This yields the process
$$W_t=\int_0^te^{(t-s)diag(-\lambda,\cdots,-\lambda)+\ln\left(\frac{e^{2\lambda t}-1+\lambda c^2}{e^{2\lambda s}-1+\lambda c^2}\right)A}dB_s   \quad(t\ge0).$$
Now we may proceed  as in Lemma \ref{cov-matrix-W} to calculate the covariance matrix.
\end{proof}

\begin{remark} The long-term behaviour of the covariance matrix is inherited by the long-term behaviour of $Y$.
 In the ergodic case  for $Y$, i.e. $\lambda>0$, we obtain
 $\lim_{t\to\infty}\Sigma^\lambda_t=\frac{1}{2\lambda}(E-A)^{-1}$.
 For $\lambda<0$ we need an exponential scaling
 $$\lim_{t\to\infty}e^{2\lambda t}\Sigma^\lambda_t=\frac{\lambda c^2-1}{2\lambda}(E-A)^{-1}(E-e^{\ln(\frac{\lambda c^2}{\lambda c^2-1})(E-A)}).$$
\end{remark}


\begin{thebibliography}{999}

\bibitem[AG]{AG} R. Allez, A. Guionnet, A diffusive matrix model for invariant $\beta$-ensembles.
\textit{Electron. J. Probab. } 18 (2013), 1-30.



\bibitem[AKM1]{AKM1} S. Andraus, M. Katori, S. Miyashita, Interacting particles on the line 
and Dunkl intertwining operator of type $A$: Application to the freezing regime. 
\textit{J. Phys. A: Math. Theor. } 45  (2012) 395201.

\bibitem[AKM2]{AKM2} S. Andraus, M. Katori, S. Miyashita, 
Two limiting regimes of interacting Bessel processes. 
 \textit{J. Phys. A: Math. Theor. } 47  (2014) 235201.

\bibitem[AM]{AM} S. Andraus, S. Miyashita, Two-step  asymptotics of scaled Dunkl processes.
\textit{J.  Math. Phys.} 56  (2015) 103302.


\bibitem[AV1]{AV} S. Andraus, M. Voit, Limit theorems
 for multivariate Bessel processes in the freezing regime. \textit{Stoch. Proc. Appl. }
 129 (2019), 4771-4790.


\bibitem[AV2]{AV2} S. Andraus, M. Voit, Central limit theorems for multivariate Bessel processes
 in the freezing regime II: The covariance matrices.
\textit{J. Approx. Theory}  246 (2019), 65–84. 



\bibitem[A]{A} J.-P. Anker.  An introduction to Dunkl theory and its analytic aspects. In: G. Filipuk,
Y. Haraoka, S. Michalik. Analytic, Algebraic and Geometric Aspects of Differential Equations,
Birkh{\"a}user, pp.3-58, 2017.

 \bibitem[BF]{BF} T.H. Baker, P.J. Forrester, The Calogero-Sutherland model and generalized classical
polynomials. \textit{Comm. Math. Phys.} 188 (1997), 175--216.


\bibitem[CGY]{CGY} O. Chybiryakov, L. Gallardo, M. Yor, Dunkl processes and their radial parts relative to a root system. In:
 P. Graczyk et al. (eds.), Harmonic and stochastic analysis of Dunkl processes. Hermann, Paris 2008.

%\bibitem[D]{D} P. Deift, Orthogonal Polynomials and Random Matrices: A Riemann-Hilbert Approach. Amer. Math. Soc. 2000.


\bibitem[Dem]{Dem} N. Demni, Generalized Bessel function of type D. 	
SIGMA 4 (2008), 075, 7 pages, arXiv:0811.0507. 

\bibitem[DF]{DF} P. Desrosiers, P. Forrester,
 Hermite and Laguerre $\beta$-ensembles: Asymptotic corrections to the eigenvalue density. \textit{Nuclear Physics B}
743 (2006), 307-332.

\bibitem[DV]{DV} J.F. van Diejen, L. Vinet, Calogero-Sutherland-Moser Models.
 CRM Series in Mathematical Physics, Springer, Berlin, 2000.



\bibitem[DE1]{DE1} I. Dumitriu, A. Edelman, Matrix models for beta-ensembles. \textit{J. Math. Phys.} 43 (2002),  5830-5847.

\bibitem[DE2]{DE2} I. Dumitriu, A. Edelman, Eigenvalues of Hermite and Laguerre ensembles: large beta asymptotics,
\textit{Ann. Inst. Henri Poincare (B)} 41 (2005), 1083-1099.


\bibitem[GY]{GY} L. Gallardo, M. Yor, Some remarkable properties of the Dunkl martingale. In:
 Seminaire de Probabilites XXXIX, pp. 337-356, dedicated to P.A. Meyer, vol. 1874,
 Lecture Notes in Mathematics, Springer, Berlin, 2006.

\bibitem[GM]{GM} P.  Graczyk, J. Malecki, Strong solutions of non-colliding particle systems.
\textit{ Electron. J. Probab.} 19 (2014), 21 pp.

%\bibitem[HS]{HS} M.W. Hirsch, S. Smale, Differential Equations, Dynamical Systems, and Linear Algebra.
%Academic Press, San Diego, CA, 1974.

\bibitem[HP]{HP} D. Holcomb, E. Paquette, Tridigonal models for Dyson Brownian motion. Preprint 2017, arXiv 1707.02700v1.


%\bibitem[Me]{Me} M. Mehta, Random matrices (3rd ed.), Elsevier/Academic Press, Amsterdam, 2004. 

\bibitem[P]{P} P.E. Protter, Stochastic Integration and Differential Equations. A New Approach.
 Springer, Berlin, 2003.


\bibitem[R1]{R1} M. R\"osler,
Generalized Hermite polynomials and the heat equation for Dunkl operators.
\textit{Comm. Math. Phys.} 192 (1998),  519-542.

\bibitem[R2]{R2} M. R\"osler, Dunkl operators: Theory and applications.
In: Orthogonal polynomials and special functions, Leuven 2002, \textit{Lecture Notes in Math.} 1817 (2003), 93--135.

\bibitem[RV1]{RV1} M. R\"osler, M. Voit, Markov processes related with Dunkl operators.
\textit{Adv. Appl. Math.}  21 (1998) 575-643.

\bibitem[RV2]{RV2} M. R\"osler, M. Voit, Dunkl theory, convolution algebras, and related Markov processes.
 In: P. Graczyk et al. (eds.), Harmonic and stochastic analysis of Dunkl processes. Hermann, Paris 2008.

\bibitem[Sch]{Sch} B. Schapira,
The Heckman-Opdam Markov processes.
 \textit{Probab. Theory Rel. Fields} 138 (2007), 
 495-519.

\bibitem[S]{S} G. Szeg{\"o}, Orthogonal Polynomials. 
Colloquium Publications (American Mathematical Society), Providence, 1939.


\bibitem[V]{V} M. Voit, Central limit theorems 
 for multivariate Bessel processes in the freezing regime.\textit{ J. Approx. Theory } 239 (2019), 210--231.
 
\bibitem[VW]{VW} M. Voit, J.H.C. Woerner, 
The differential equations associated with Calogero-Moser-Sutherland particle models 
in the freezing regime. Preprint 2019.

\end{thebibliography}
\end{document}